\documentclass[11pt]{amsart}
\usepackage[T1]{fontenc}
\usepackage[utf8]{inputenc}
\usepackage{lmodern,amsmath,amssymb,mathtools,microtype}
\usepackage[margin=1.08in]{geometry}
\usepackage[hidelinks]{hyperref}
\hypersetup{pdftitle={Quantum Laurent positivity in rank three},pdfsubject={Quantum Laurent positivity}}

\numberwithin{equation}{section}
\newtheorem{theorem}{Theorem}[section]
\newtheorem{lemma}[theorem]{Lemma}
\newtheorem{proposition}[theorem]{Proposition}
\newtheorem{corollary}[theorem]{Corollary}
\theoremstyle{definition}

\newcommand{\Z}{\mathbb Z}
\newcommand{\N}{\mathbb N}
\newcommand{\Q}{\mathbb Q}
\newcommand{\bt}{\widetilde B}
\newcommand{\diag}{\operatorname{diag}}
\newcommand{\Mat}{\operatorname{Mat}}
\newcommand{\ee}{\mathbf e}
\title[Quantum Laurent positivity in rank three]{Quantum Laurent positivity in rank three}
\author{Qiyue Tang}
\address{Zhili College, Tsinghua University, Beijing 100084, China}
\email{tangqy24@mails.tsinghua.edu.cn}
\begin{document}
\begin{abstract}
We prove Laurent positivity for skew-symmetrizable quantum cluster algebras of rank three, with arbitrary integral compatible quantizations and invertible frozen variables. The proof uses mutation-acyclic positivity from a companion manuscript.
\end{abstract}
\maketitle

\section{Introduction}

Cluster algebras were introduced by Fomin and Zelevinsky as an
algebraic framework related to total positivity and canonical bases
\cite[Introduction]{FZ1}. Their generators, called cluster variables,
are grouped into overlapping clusters connected by mutation. Although
mutation is defined by rational exchange relations, every cluster
variable is a Laurent polynomial in any chosen cluster
\cite[Theorem~3.1]{FZ1}. The positivity conjecture asserts that the
coefficients of these expansions are nonnegative; it was formulated
immediately after that theorem.

Lee and Schiffler proved positivity for skew-symmetric coefficient-free
cluster algebras of rank three \cite[Theorem~1.1]{LS3}, using rank-two
expansion formulas and induction along mutation paths. They subsequently
proved positivity for skew-symmetric cluster algebras of arbitrary rank
\cite[Theorem~1.1]{LS}. Gross, Hacking, Keel and Kontsevich established
classical Laurent positivity for skew-symmetrizable cluster algebras of
geometric type through scattering diagrams and broken lines
\cite[Theorem~4.10]{GHKK}.

Berenstein and Zelevinsky introduced quantum cluster algebras by
replacing each commutative cluster with a family of quasi-commuting
variables \cite{BZ}. Their quantum Laurent phenomenon
\cite[Corollary~5.2]{BZ} places every cluster variable in the based
quantum torus of every seed. Quantum Laurent positivity asks that its
coefficients in the normalized monomial basis lie in
$\N[q^{\pm1/2}]$. Davison proved this for skew-symmetric quantum cluster
algebras by cohomological methods, using purity of the mixed Hodge
structures arising from quivers with potential
\cite[Theorem~2.4]{Davison}.

The rank is the number of mutable directions. We use the
Berenstein--Zelevinsky convention over $\Z[q^{\pm1/2}]$, where
$q^{1/2}$ is a formal indeterminate.
A seed consists of a toric frame $M:\Z^m\to\mathcal F\setminus\{0\}$ and an
integer exchange matrix $\bt\in\Mat_{m\times n}(\Z)$ whose principal
part $B$ is skew-symmetrizable. The skew field $\mathcal F$ contains
the based quantum torus of the frame. Its integral skew-symmetric
form $\Lambda$ satisfies
\begin{equation}\label{eq:compatibility}
 \bt^T\Lambda=(D\ \ 0),\qquad
 D=\diag(d_1,\ldots,d_n),\quad d_i>0,
\end{equation}
and the based monomials multiply by
\begin{equation}\label{eq:torus}
 M(a)M(b)=q^{\frac12 a^T\Lambda b}M(a+b).
\end{equation}
The generators $X_i=M(\ee_i)$ with $i\le n$ are mutable; the others
are frozen and invertible. Precise mutation conventions are recalled
in Section~\ref{sec:prelim}.

Write $\Sigma_t=(M_t,\bt_t)$ for the seeds of a quantum cluster
pattern indexed by the labelled $n$-regular tree $\mathbb T_n$, and
$X_{j;t}=M_t(\ee_j)$, $1\le j\le n$, for its cluster variables. Fix an arbitrary vertex
$t_0$ as the initial vertex.

\begin{theorem}\label{thm:A}
Suppose that $n=3$. For every integral compatible initial seed
$\Sigma_{t_0}$, every $t\in\mathbb T_3$ and $j\in\{1,2,3\}$, there is
a finite expansion
\[
 X_{j;t}=\sum_{h\in\Z^m}c_{j;t,h}^{t_0}(q^{1/2})M_{t_0}(h),
 \qquad c_{j;t,h}^{t_0}(q^{1/2})\in\N[q^{\pm1/2}].
\]
\end{theorem}

Since $t_0$ is arbitrary, the theorem gives positivity in every seed.

A principal matrix is \emph{mutation-acyclic} if its mutation class
contains an acyclic oriented graph; otherwise it is called
\emph{mutation-cyclic}. We use the mutation-acyclic
positivity theorem of the companion manuscript
\cite[Theorem~1.1]{MA}, which is not yet publicly available.
We recall this result as Lemma~\ref{thm:acyclic-input},
which settles the mutation-acyclic case of Theorem~\ref{thm:A}. By freezing one mutable direction, it also yields the rank-two positivity needed in our proof of the mutation-cyclic case.

In the mutation-cyclic case, we use a tropical support inequality to decompose a Laurent expression into three positive parts, supported in adjacent seeds and with nonnegative exponents in the two active directions. We then argue by strong induction on the mutation word length. The induction hypothesis gives positivity in the three auxiliary seeds, while rank-two positivity carries the decomposition along the final two-direction segment.

\medskip
\noindent\textbf{Use of generative AI.}
The main results of this paper were generated using GPT-6-Astra and the
Danus system. The author supplied relevant references, checked the
manuscript, and revised its exposition.

\section{Quantum seeds and mutation-acyclic positivity}
\label{sec:prelim}

\subsection{Based quantum tori and seeds}
\leavevmode\par
We recall the quantum seed conventions used throughout the paper.

We use $\N=\{0,1,2,\ldots\}$ and write $\ee_1,\ldots,\ee_m$ for the
standard basis of $\Z^m$. For an integral skew-symmetric matrix
$\Lambda_0$, the based quantum torus $\mathcal T_{\Lambda_0}$ is the
free $\Z[q^{\pm1/2}]$-module with basis $\{X^h:h\in\Z^m\}$ and product
\[
 X^hX^u=q^{h^T\Lambda_0u/2}X^{h+u}.
\]
It is an Ore domain, with fraction division ring $\mathcal F$
\cite[Definition~4.1 and Section~11]{BZ}.

A toric frame is a map $M:\Z^m\to\mathcal F\setminus\{0\}$ of the
form $M(h)=\varphi(X^{\eta(h)})$, where $\eta:\Z^m\to\Z^m$ is a
lattice isomorphism and $\varphi$ is a $\Q(q^{1/2})$-algebra
automorphism of $\mathcal F$ \cite[Definition~4.3]{BZ}.
Its associated form is
$\Lambda_M(h,u)=\eta(h)^T\Lambda_0\eta(u)$.
When the frame is understood, we write $\Lambda=(\lambda_{ij})$
for the matrix of $\Lambda_M$. Thus
\[
 X_iX_j=q^{\lambda_{ij}}X_jX_i,\qquad X_i=M(\ee_i).
\]

A quantum seed is a pair $(M,\bt)$, where
$\bt=(b_{ij})\in\Mat_{m\times n}(\Z)$ and the matrix $\Lambda$ of
the frame satisfies \eqref{eq:compatibility}.
Its principal part $B$ consists of the first
$n$ rows of $\bt$. Compatibility implies that
$DB=\bt^T\Lambda\bt$ is skew-symmetric, so $B$ is
skew-symmetrizable. We use the oriented graph with an arrow
$i\to j$ when $b_{ij}>0$; acyclicity refers to this graph.

A quantum cluster pattern assigns a seed to each vertex of the
labelled $n$-regular tree $\mathbb T_n$, whose incident edges at each
vertex have distinct labels $1,\ldots,n$. Adjacent seeds are related
by mutation in the edge direction. The frozen generators
$X_{n+1},\ldots,X_m$ are unchanged throughout the pattern.
The quantum cluster algebra $\mathcal A_q$ is the
$\Z[q^{\pm1/2}]$-subalgebra of $\mathcal F$ generated by all cluster
variables and by the frozen generators and their inverses
\cite[Definition~4.12]{BZ}.

\subsection{Mutation and positive Laurent expressions}
\leavevmode\par
We recall the mutation rules and introduce the Laurent expansion
notation used in the proofs.

The based monomials of a toric frame are linearly independent over
$\Z[q^{\pm1/2}]$. Following \cite[Sections~3--4]{BZ}, they are
normalized by
\[
 M(h)=q^{-\frac12\sum_{i<j}h_i h_j\lambda_{ij}}
          X_1^{h_1}\cdots X_m^{h_m}.
\]
For $z\in\Z$, write $[z]_+=\max(z,0)$, and apply this notation
componentwise to vectors. Let $\bt_k$ be the $k$th column of $\bt$.
Mutation in direction $1\le k\le n$ replaces $X_k$ by
\begin{equation}\label{eq:mutation}
 X_k'=M(-\ee_k+[\bt_k]_+)+M(-\ee_k+[-\bt_k]_+),
\end{equation}
and retains the other generators. The exchange matrix mutates by
\[
 b_{ij}'=
 \begin{cases}
 -b_{ij},&i=k\text{ or }j=k,\\
 b_{ij}+[b_{ik}]_+b_{kj}+b_{ik}[-b_{kj}]_+,&\text{otherwise}.
 \end{cases}
\]
The form mutates with the frame according to the compatible-pair
rules of \cite[Sections~3--4]{BZ}. Mutations are involutions.

For a seed $W$ of the pattern, denote its toric frame by $M_W$ and
its based quantum torus by $\mathcal T_W$. Write $W[k]$ for its
mutation in direction $k$. A mutation path is \emph{reduced} if
consecutive labels are distinct.

By \cite[Corollary~5.2]{BZ}, every cluster variable of the pattern
belongs to $\mathcal T_W$ for every seed $W$.
For any $F\in\mathcal T_W$, write its unique finite expansion as
\[
 F=\sum_{h\in\Z^m}c_h(q^{1/2})M_W(h),\qquad
 c_h(q^{1/2})\in\Z[q^{\pm1/2}].
\]
Its support in $W$ is
$\operatorname{supp}_W(F)=\{h:c_h(q^{1/2})\ne0\}$.
We call $F$ \emph{positive in $W$} if every coefficient belongs to
$\N[q^{\pm1/2}]$; this includes $F=0$.
By \eqref{eq:torus}, sums and products of positive Laurent
polynomials in a fixed based torus are positive.

For $1\le a\le m$ and $I\subseteq\Z$, define the degree part
\[
 [F]_{W,a\in I}
 =\sum_{\substack{h\in\Z^m\\h_a\in I}}
 c_h(q^{1/2})M_W(h),
\]
where $h_a$ is the $a$th coordinate of $h$.
We write $[F]_{W,a>0}$ when $I=\Z_{>0}$, and similarly for other
ranges of exponents.

The specialization $\mathcal T_W/(q^{1/2}-1)\mathcal T_W$ is a
commutative Laurent polynomial ring, and the quantum exchange
relations specialize to the classical ones. If $F$ is positive in
$W$, each nonzero coefficient satisfies $c_h(1)>0$.
Thus, after this specialization, setting the original frozen
variables equal to $1$ preserves the mutable support: the projection
of $\operatorname{supp}_W(F)$ onto the first $n$ coordinates.

\subsection{Mutation-acyclic positivity and freezing}
\leavevmode\par
We use the following result to obtain positivity in the rank-two
patterns that occur in the proof.

\begin{lemma}[{\cite[Theorem~1.1]{MA}}]
\label{thm:acyclic-input}
For a mutation-acyclic skew-symmetrizable quantum cluster pattern of
arbitrary finite mutable rank, every cluster variable has a Laurent
expansion with coefficients in $\N[q^{\pm1/2}]$ in every seed, for every
integral compatible quantization and with invertible frozen variables.
\end{lemma}

To freeze the mutable directions outside $J\subseteq\{1,\ldots,n\}$,
keep all rows of $\bt$, retain only the columns indexed by $J$, and
leave the form unchanged. Ordering $J$ first in the coordinate
lattice gives the compatibility equation $(D_J\ \ 0)$ for the
restricted exchange matrix, where $D_J=\diag(d_j:j\in J)$ in the
chosen order. Mutations in $J$ agree with the original mutations;
the other generators are now invertible frozen variables.

\begin{samepage}
\begin{corollary}\label{cor:rank-two}
Fix two distinct mutable directions $a,b$ and freeze every other
mutable direction. Every cluster variable of this pair pattern is
positive in every seed of the pair pattern. Moreover, a based monomial
in any seed of the pair pattern with nonnegative $a,b$ exponents and
arbitrary integer exponents elsewhere is positive in every seed of
the pair pattern.
\end{corollary}
\end{samepage}
\begin{proof}
A skew-symmetrizable two-by-two principal matrix is acyclic, so the
first assertion follows from Lemma~\ref{thm:acyclic-input} by freezing.

For the second assertion, write the based monomial as an ordered
product of nonnegative powers of the two active variables and integer
powers of the frozen generators, multiplied by a power of $q^{1/2}$.
Expand the active variables in the target seed using the first
assertion. The frozen generators are unchanged, and
\eqref{eq:torus} makes each resulting product positive.
\end{proof}

\section{A positive decomposition in three adjacent seeds}
\label{sec:decomposition}

We first compare the least exponent in an unchanged direction across
one mutation, then use this comparison to obtain a positive
decomposition in three adjacent seeds.
For distinct mutable directions $a,b$, a based monomial is
\emph{localized in the pair $a,b$} if its $a,b$ exponents are
nonnegative; the other exponents may be arbitrary integers.
For $0\ne F\in\mathcal T_W$, its least exponent in direction $a$ is
$\min\{h_a:h\in\operatorname{supp}_W(F)\}$.

\begin{lemma}\label{lem:min}
Suppose $F\ne0$ is Laurent in a seed $W$ and its mutation $W[b]$.
For an unchanged direction $a\ne b$, the least $a$ exponent of $F$
is the same in these two seeds.
\end{lemma}
\begin{proof}
In the coordinates of $W$, let $K$ be the fraction division ring of
the quantum torus generated by $X_j$, $j\notin\{a,b\}$, over
$\Q(q^{1/2})$. Conjugation by $X_a$ defines an automorphism $\sigma$
of the division subring $K(X_b)$ generated by $K$ and $X_b$.
We use the $X_a$-adic valuation on the skew rational function field
$K(X_b)(X_a;\sigma)$, viewed inside the skew Laurent series division
ring $K(X_b)((X_a;\sigma))$.

Ordering the exchange relation gives
\[
 X_b'=Q(X_a)X_b^{-1},\qquad Q(X_a)\in K[X_a;\sigma].
\]
The two exchange monomials have nonnegative $X_a$ degrees, at least
one of which is zero. Their positive coefficients give a nonzero
constant term $C\in K^\times$. Thus $X_b'$ has valuation zero and
initial coefficient $CX_b^{-1}$.
For every $r\in\Z$, the initial coefficient of $(X_b')^r$ is
$C_rX_b^{-r}$ for some $C_r\in K^\times$.

If $f(X_b')=\sum_r k_r(X_b')^r$ is a nonzero finite Laurent
polynomial in $X_b'$ over $K$, its constant coefficient as an
$X_a$-series is $\sum_r k_rC_rX_b^{-r}$. This is nonzero because
distinct powers of $X_b$ are linearly independent over $K$.
Hence every such $f(X_b')$ has valuation zero.
Group the Laurent expansion of $F$ in $W[b]$ as
$F=\sum_j f_j(X_b')X_a^j$. Its valuation is the least $j$ with
$f_j\ne0$. In $W$, the same valuation is the least $X_a$ exponent,
proving the assertion.
\end{proof}

\begin{lemma}\label{lem:glue}
Let $W$ be a seed and $a,b$ distinct mutable directions.
Suppose $F$ is positive in $W,W[a],W[b]$, and every
$h\in\operatorname{supp}_{W[a]}(F)$ with $h_a>0$ satisfies $h_b\ge0$.
Then
\begin{equation}\label{eq:three}
 F=P+R_0+R_1,
\end{equation}
where $P,R_0,R_1$ are finite positive sums of localized $a,b$
monomials in $W[a],W,W[b]$, respectively. Consequently, $F$ is positive
in every seed of the $a,b$ pattern through $W$.
\end{lemma}
\begin{proof}
The case $F=0$ is immediate. Suppose $F\ne0$.
For $k\in\{a,b\}$, let $K_k$ be the division subring of
$\mathcal F$ generated over $\Q(q^{1/2})$ by the variables
$X_j=M_W(\ee_j)$ with $j\ne k$. Conjugation by $X_k$
restricts to an automorphism $\sigma_k$ of $K_k$. The exchange
relation has the form
\[
 X_k'=A_kX_k^{-1},\qquad A_k\in K_k^\times.
\]
Indeed, $b_{kk}=0$, and factoring $X_k^{-1}$ on the right leaves
two monomials in the unchanged variables whose coefficients are
powers of $q^{1/2}$. Their sum $A_k$ is therefore nonzero.
Both $\mathcal T_W$ and $\mathcal T_{W[k]}$ are contained in
the skew Laurent ring $K_k[X_k^{\pm1};\sigma_k]$, where
$X_ku=\sigma_k(u)X_k$ for $u\in K_k$.
The powers of $X_k$ are linearly independent over $K_k$:
after clearing a common left denominator in the quantum torus
generated by the unchanged variables, this follows from the
linear independence of the based monomials in $W$.
Thus this ring has a direct $\Z$-grading with
$\deg K_k=0$, $\deg X_k=1$, and $\deg X_k'=-1$.
For every $G\in\mathcal T_W\cap\mathcal T_{W[k]}$ and
$I\subseteq\Z$, uniqueness of the finite homogeneous
decomposition therefore gives, as an equality in $\mathcal F$,
\[
 [G]_{W[k],k\in I}=[G]_{W,k\in -I},
 \qquad -I=\{-r:r\in I\}.
\]
In particular, taking $k=a$, set
\[
 P=[F]_{W[a],a>0}=[F]_{W,a<0},\qquad
 R=F-P=[F]_{W,a\ge0}.
\]
The support hypothesis makes every monomial of $P$ in $W[a]$ localized
in $a,b$. By Corollary~\ref{cor:rank-two}, $P$ is positive throughout
the pair pattern. The remainder $R$ is positive in $W$ and $W[a]$, is
Laurent in $W[b]$, and has nonnegative $a$ exponents in $W$.
If $R=0$, take $R_0=R_1=0$; assume now that $R\ne0$.

Expanding the positive powers of the mutated $a$ generator shows
that $P$ has nonnegative $b$ exponents in $W$. Reversing the $b$
mutation gives only nonpositive exponents in direction $b$ of $W[b]$.
It follows that
\[
 [R]_{W[b],b>0}=[F]_{W[b],b>0},\qquad
 [R]_{W[b],b\le0}=[R]_{W,b\ge0}.
\]
The first part is positive in $W[b]$. For the second part, expand
the nonnegative powers of the $b$ generator in $W$ by the reverse
exchange formula; this gives a positive expression in $W[b]$.
Thus $R$ is positive in $W[b]$. Lemma~\ref{lem:min} shows that its
$a$ exponents there are nonnegative.

Set
\[
 R_0=[R]_{W,b\ge0},\qquad R_1=[R]_{W[b],b>0}.
\]
The preceding degree identities give $R=R_0+R_1$.
Both pieces have nonnegative $a,b$ exponents in their respective
seeds, so \eqref{eq:three} follows.
Corollary~\ref{cor:rank-two} carries all three pieces through the
pair pattern.
\end{proof}

\section{Support separation along a directed three-cycle path}
\label{sec:separation}

We now establish the support condition needed in
Lemma~\ref{lem:glue} for a path in a mutation-cyclic rank-three pattern.

\begin{proposition}\label{prop:separator}
Suppose every principal matrix in the mutation class has a directed
three-cycle as its oriented graph. Let a reduced path of length at
least two run from $S_0$ to $S$, with first label $i_0$ and last two
labels $c,a$, and let $b$ be the remaining mutable label.
If $M_{S_0}(\ee_{i_0})$ is positive in $S$, then every
$h\in\operatorname{supp}_S(M_{S_0}(\ee_{i_0}))$ with $h_a>0$ satisfies $h_b>0$.
\end{proposition}
\begin{proof}
Give the original frozen generators weight zero. We first construct
a row vector $w$ of positive rational homogeneous weights on the
three mutable directions.
At a cyclic seed, order the mutable labels as $(k,i,j)$ so that
\[
 B=\begin{pmatrix}0&d&-f\\-d'&0&e\\f'&-e'&0\end{pmatrix},
 \qquad d,d',e,e',f,f'>0.
\]
Symmetrizability gives $def'=d'e'f$.
Choose $w_i\in\Q_{>0}$ and set
$w_j=(d'/f')w_i$, $w_k=(e/f)w_i$. Then
\[
 -d'w_i+f'w_j=0,\qquad dw_k-e'w_j=0,\qquad -fw_k+ew_i=0,
\]
so $wB=0$. The two mutable exchange powers at $k$ are
$p_k=d'$ and $r_k=f'$, with common weight
\[
 \kappa=p_kw_i=r_kw_j.
\]
Under mutation at $k$, put $w_k'=\kappa-w_k$ and leave the other
weights unchanged. The opposite entry becomes $e-d'f<0$, since
the mutated graph is again a directed three-cycle. Hence
$w_k'=(d'-e/f)w_i>0$.
For $B'=\mu_k(B)$, we have
\[
 \begin{aligned}
 -dw_k'+(df'-e')w_j&=dw_k-e'w_j=0,\\
 fw_k'+(e-d'f)w_i&=-fw_k+ew_i=0,
 \end{aligned}
\]
and $d'w_i-f'w_j=0$, so $w'B'=0$.
Starting with this construction at $S_0$, we obtain positive
homogeneous weights along the path; denote the initial vector by $w^0$.

Next take the mutable weight vector $\xi^0$ with
$\xi^0_{i_0}=-1$ and the other two coordinates zero, and transport
$\xi$ by tropical mutation:
\[
 \xi_k'=\max(p_k\xi_i,r_k\xi_j)-\xi_k.
\]
Here $i,j,p_k,r_k$ are taken at the current seed, and unchanged
coordinates retain their values. For each mutable label $u$, write
$\rho_u=\xi_u/w_u$ and $\rho_{\max}=\max_u\rho_u$.
If the maximum occurs away from $k$, then
\begin{equation}\label{eq:ratios}
 \rho_k'=\frac{\kappa\rho_{\max}-w_k\rho_k}{\kappa-w_k}
 =\rho_{\max}+\frac{w_k}{w_k'}(\rho_{\max}-\rho_k).
\end{equation}
The first mutation makes $i_0$ the unique positive maximum.
Each subsequent mutation is in a different direction from the
previous one, and \eqref{eq:ratios} gives a new strict maximum.
At $S$, therefore,
\[
 \rho_a>\rho_c>\rho_b\ge0,\qquad \rho_c>0.
\]
Using these terminal values, set
\[
 s=\frac1{w_a(\rho_a-\rho_c)},\qquad
 \tau=s\rho_c,\qquad \lambda=sw_b(\rho_c-\rho_b)>0.
\]
The weights $\omega=-\tau w+s\xi$ obey tropical mutation: the
two $w$-weights in each exchange relation agree, and $s>0$, so their
common contribution can be taken outside the maximum.
At $S$ this gives
\[
 \omega_a=1,\qquad \omega_c=0,\qquad \omega_b=-\lambda,
\]
while at $S_0$ the source generator has weight
$-\tau w_{i_0}^0-s<0$.

Let $x_1,x_2,x_3$ be independent commuting variables, and consider
the ring homomorphism
\[
 \phi:\mathcal T_S\longrightarrow
 \Z[x_1^{\pm1},x_2^{\pm1},x_3^{\pm1}],\qquad
 \phi(q^{1/2})=1,\qquad
 \phi(M_S(h))=x_1^{h_1}x_2^{h_2}x_3^{h_3}.
\]
Thus $\phi$ specializes $q^{1/2}$ and all original frozen
variables to $1$. By the Laurent phenomenon, every cluster
variable along the path belongs to $\mathcal T_S$.
At any step, multiply the quantum exchange relation on the left
by the variable being mutated. The resulting identity involves
only nonnegative powers of the current seed variables and powers
of $q^{1/2}$, so it can be specialized by $\phi$.
It becomes the classical exchange identity with the original
frozen variables set to $1$.

Starting at $S$ and proceeding backwards, the specialized current
variables are nonzero subtraction-free rational functions of
$x_1,x_2,x_3$. This holds initially, and the specialized exchange
identity determines the next variable by division by the current
variable in $\Q(x_1,x_2,x_3)$; the result is again a nonzero
subtraction-free rational function. Induction therefore identifies
the images under $\phi$ of all variables along the path with their
classical counterparts. 

Now evaluate $x_u=T^{\omega_u}$ for $T>0$ and follow the
classical path backwards. Every resulting variable is a positive
function of $T$ with an asymptotic expression
$f(T)=cT^\alpha(1+o(1))$ as $T\to+\infty$, where $c>0$.
Indeed, the terminal variables have this form, and this class of
functions is closed under addition, multiplication, and division.
For such a function define
\[
 \nu_\infty(f)=\lim_{T\to+\infty}
 \frac{\log f(T)}{\log T}=\alpha.
\]
For functions of this form,
\[
 \begin{aligned}
 \nu_\infty(fg)&=\nu_\infty(f)+\nu_\infty(g),\\
 \nu_\infty(f/g)&=\nu_\infty(f)-\nu_\infty(g),\\
 \nu_\infty(f+g)&=\max\{\nu_\infty(f),\nu_\infty(g)\}.
 \end{aligned}
\]
In the last identity, equal exponents cause no cancellation
because the leading coefficients are positive. Hence the growth
exponents obey the max tropical mutation rule. This rule is
involutive together with exchange-matrix mutation, so the backwards
recursion recovers the weight vector $\omega$ at each seed along
the path. In particular, the growth exponent of the specialized
source generator is $-\tau w_{i_0}^0-s$.

Since the source generator is positive in $S$, specialization preserves its mutable
support, and this growth exponent is also the maximum of
$h_a-\lambda h_b$ over $h\in\operatorname{supp}_S(M_{S_0}(\ee_{i_0}))$.
Consequently every such $h$ satisfies
\begin{equation}\label{eq:separator}
 h_a-\lambda h_b\le-\tau w_{i_0}^0-s<0.
\end{equation}
In particular, $h_a>0$ implies $h_b>0$.
\end{proof}

\section{Induction on the full mutation word}
\label{sec:induction}

We combine the support separation and the three-seed decomposition
to prove Theorem~\ref{thm:A}.

\begin{proof}[Proof of Theorem~\ref{thm:A}]
If the principal matrix is mutation-acyclic, the result follows from
Lemma~\ref{thm:acyclic-input}. Otherwise every principal matrix in
its mutation class has a directed three-cycle as its oriented graph:
on three vertices, this is the only possible oriented cycle.

In this case, we prove positivity simultaneously for every source
seed and source generator by strong induction on the mutation word
length $N$. The empty word gives a based monomial.
Consecutive repeated labels cancel, so we may assume the word is
reduced. If the first mutation leaves the source generator unchanged,
move the source seed across that edge and apply induction to the
shorter word. We may therefore assume that the first label is the
source direction.

A word involving at most two labels is covered by
Corollary~\ref{cor:rank-two}; a source generator outside those
directions stays unchanged. Suppose now that all three labels occur.
Choose the maximal terminal subword using the last two labels, and
write its length as $\ell\ge2$. The preceding prefix is nonempty
and ends in the third label $c$. Let $W$ be its terminal seed, and
call the first label of the suffix $a$ and the other label $b$.
For $r=N-\ell$, the paths from the source to
\[
 W,\qquad W[a],\qquad W[b]
\]
have lengths $r,r+1,r+1<N$. By induction, the source generator is positive in all three seeds.

The path to $W[a]$ starts in the source direction and ends in $c,a$.
It is reduced and has length at least two, so
Proposition~\ref{prop:separator} gives
$h_a>0\Rightarrow h_b>0$ for every supported exponent of the source generator in $W[a]$.
Lemma~\ref{lem:glue} now expresses this generator as three positive sums of
localized $a,b$ monomials. Corollary~\ref{cor:rank-two} carries these
sums through the terminal two-direction segment, proving positivity
at its endpoint and completing the induction.

For the variable $X_{j;t}$ in the statement, take $\Sigma_t$ as the
source seed and follow the path to $t_0$.
\end{proof}

\end{document}